\documentclass[12pt]{amsart}

\usepackage{latexsym}
\usepackage{amssymb}
\usepackage[cp850]{inputenc}
\usepackage{amsthm}
\usepackage{amscd}
\usepackage{amsmath}
\usepackage{amsfonts}
\usepackage[all]{xy}
\usepackage[pdftex]{color}
\usepackage[pdftex]{graphicx}

\newtheorem{sat}{Theorem}				
\newtheorem{lem}{Lemma}
\newtheorem*{lem*}{Lemma}
\newtheorem{kor}{Corollary}			\newtheorem{prop}{Proposition}
				
\newtheorem*{defi*}{Definition}	\newtheorem*{bei*}{Example}
\newtheorem*{sat*}{Theorem}			\newtheorem*{kor*}{Corollary}
\newtheorem*{rmk*}{Remark}
\newtheorem*{conj*}{Conjecture}
\newtheorem*{prop*}{Proposition}
\newtheorem*{question*}{Question}

\let\ssection=\section
\renewcommand{\section}{\setcounter{equation}{0}\ssection}

\newtheorem*{namedtheorem}{\theoremname}
\newcommand{\theoremname}{testing}
\newtheorem*{threegap}{Three Gap Theorem}
\newtheorem*{theoremone}{Theorem 1}
\newtheorem*{theoremto}{Corollary 1}
\newtheorem*{theoremthree}{Theorem 2}

\theoremstyle{remark}

\newcommand{\BR}{\mathbb R}			
\newcommand{\BN}{\mathbb N}			
\newcommand{\BS}{\mathbb S}			
				\newcommand{\BT}{\mathbb T}
\newcommand{\BP}{\mathbb P}
\newcommand{\BE}{\mathbb E}

\newcommand{\comment}[1]{}

\newcommand{\DD}{\nabla}

\DeclareMathOperator{\Isom}{Isom}	

\DeclareMathOperator{\inj}{inj}
\DeclareMathOperator{\diam}{diam}

\DeclareMathOperator{\NND}{NND}
\DeclareMathOperator{\nnd}{nnd}

\DeclareMathOperator{\dist}{d}
\DeclareMathOperator{\derivative}{D}
\begin{document}

\title[]{The Three Gap Theorem and Riemannian Geometry}

\author{Ian Biringer \& Benjamin Schmidt}

\comment{
\begin{abstract}
The classical Three Gap Theorem asserts that for $n \in \BN $ and $p \in \BR $, there are at most three distinct distances between consecutive elements in the subset of $[0,1) $ consisting of the reductions modulo $1 $ of the first $n $ multiples of $p.$  Regarding it as a statement about rotations of the circle, we find results in a similar spirit pertaining to isometries of compact Riemannian manifolds and the distribution of points along their geodesics.  
\end{abstract}
}

\maketitle

The classical Three Gap Theorem asserts that for $n \in \BN $ and $p \in \BR $, there are at most three distinct distances between consecutive elements in the subset of $[0,1) $ consisting of the reductions modulo $1 $ of the first $n $ multiples of $p $ (see e.g. \cite{So}, \cite{Sw}).  There are several interesting generalizations of this theorem in  \cite{FrSo}, \cite{Vi}, and the references therein.  Regarding it as a statement about rotations of the circle, we find results in a similar spirit pertaining to isometries of compact Riemannian manifolds and the distribution of points along their geodesics.

Let $M^k$ denote a complete Riemannian $k$-manifold with associated distance function $\dist:M \times M \rightarrow \BR$, and let $X $ be a finite subset of $M$ with $|X|\geq2$.  For $x \in X $, define the distance from $x $ to its nearest neighbor in $X $ to be $\nnd (x, X) =\min_{y \in X \setminus \{ x\} } \dist (x,y) $.  We denote the set of all nearest neighbor distances in $X $ by $$\NND (X) =\{\nnd (x, X) \ | \ x \in X\}. $$  

The following variant of the Three Gap Theorem is the starting point for our work:

\begin{threegap}[Geometric Version]\label{3gap}
Let $S^1$ denote the unit circle.  For any rotation $R$, point $p \in S^1$, and $n \in \BN$, $$|\NND(\{R^{i}(p)\,\vert\, i =0,\ldots, n\})| \leq 3.$$
\end{threegap}

Our first result shows that this phenomenon is common to all isometries of compact Riemannian manifolds with bounds depending only on dimension, sectional curvatures, and diameter.

\begin{sat}\label{iso1}
For each $k \in \BN$, $\kappa \in \BR$ and $D>0$, there is a constant $K(k,\kappa,D)\in \BN$ such that for any complete Riemannian $k$-manifold $M^k$ with $\sec \geq \kappa$ and $\diam(M)\leq D,$ and for any $I\in \Isom(M)$, $p \in M$, and $n \in \BN$, $$|\NND(\{I^{i}(p)\,\vert\, i =0,\ldots, n \})| \leq K.$$  
\end{sat}

A cleaner statement can be made for manifolds with nonnegative sectional curvatures.

\begin{kor}\label{iso2}
Let $M$ be a complete Riemannian $k$-manifold with non-negative sectional curvatures.  Then for any $I \in \Isom(M)$, $p \in M$, and $n \in \BN$, $$|\NND(\{I^{i}(p)\,\vert\,i =0,\ldots, n\})| \leq 3^k + 1.$$
\end{kor}

Working in a different direction, we also consider Riemannian metrics with uniform bounds for the number of nearest neighbor distances appearing in finite subsets of equally spaced points along geodesics.

\begin{defi*}[Bounded Geodesic Combinatorics]
Let $K\in\BN$.  A complete Riemannian manifold $M$ is defined to have \textit{K-bounded geodesic combinatorics} if for every geodesic $\gamma : \BR \to M $, $T\in\BR$, and $n \in \BN$, $$|\NND(\{\gamma (iT)\,\vert\, i =0,\ldots, n \})| \leq K.$$  
 \end{defi*}

When the particular $K\in\BN$ is irrelevant, we shall simply say that $M$ has bounded geodesic combinatorics.  By the Three Gap Theorem, the round circle has bounded geodesic combinatorics.  We will show below that all compact symmetric spaces have bounded geodesic combinatorics.  Other examples include $SC$-Riemannian manifolds, manifolds all of whose geodesics are simple closed and of common shortest period.  The Zoll type metrics on spheres provide explicit non-symmetric examples of these.  Finally, finite products of compact symmetric spaces and $SC$-metrics have bounded geodesic combinatorics.  It is possible that these are the only examples.  Our final theorem shows that this is the case among real-analytic surfaces.

\begin{sat}[Rigidity in Dimension 2]
\label{smooththeorem}
 Assume that $M$ is a Riemannian surface with bounded geodesic combinatorics.  Then $M$ is either
\begin{itemize}
\item $\BT^2 $ with a flat metric,
\item homeomorphic to $\BS^2 $ or $\BR \BP^2 $.
\end{itemize}
In the latter case, if $M$ is real-analytic or positively curved then it is either $\BR\BP^2 $ with a round metric or $\BS^2 $ with an $SC$-metric.
\end{sat}

\vskip 10pt

\textbf{Acknowledgements} Thank are due to Sujith Vijay for e-mail correspondence concerning (\cite{Vi}).  His paper provides an important idea used in the proof of Theorem \ref{iso1}.  We are also grateful to Juan Souto for the suggestion to consider asymptotic geodesics in Section 3 and to Jean-Fran\c{c}ois Lafont and Krastio Lilov for bringing the three gap theorem to the second author's attention. 

\section{Nearest Neighbors and Orbits of Isometries}
\label{isometries}
The goal of this section is to understand nearest neighbor distances in orbit segments of isometries on closed Riemannian manifolds.  We show that the number of distinct distances is limited by a certain packing number, and then use this to prove Theorem \ref{iso1}.  Finally, we construct orbit segments in $k $-dimensional flat tori with at least $k $ distinct nearest neighbor distances.

Throughout the following, let $(M,\dist) $ be a complete Riemannian $k$-manifold.  Define $P (M, r) $ to be the maximum number of points that can be packed pairwise $r $-apart into some open $r $-ball in $M $, and let $P (M) = \sup_{ r } P (M, r) $.

\begin{lem} \label{packing} For all $I \in \Isom(M)$, $p\in M$, and $n \in \BN $, we have
$$|\NND(\{I^i (p)\,\vert\,i =0,\ldots, n\})| \leq P (M) +1.$$
\end{lem}

\begin{proof}
Consider for $l \in \BN $, the set $O_l =\{ I^i (p) \ | \ i =0,\ldots, l\} $.  Our goal is to show that $|\NND (O_n) | \leq P (M) +1 $.  First, notice that
\begin{eqnarray*}
\nnd (I^i (p),O_n) & = & \min\{\dist(I^j(p),I^i(p)) \,  \vert\, I^j (p)\in O_n \setminus \{ I^i (p)\} \} \\
& = & \min\{\dist(p,I^{|i-j|}(p)) \,  \vert \, I^j (p)\in O_n \setminus \{ I^i (p)\}  \} \\
& = & \min\{\dist(p,I^j(p)) \,  \vert  \, I^j (p)\in O_{\max\{i, n -i\} } \setminus \{p\}   \}.
\end{eqnarray*}
Therefore, $\nnd (I^i (p),O_n) =\nnd (I^{ n -i } (p),O_n) $ and  $$\nnd (p, O_n) \leq \nnd (I (p),O_n) \leq \ldots \leq \nnd (I^{\lfloor \frac { n } {2} \rfloor } (p),O_n).  $$
Note that the number of distinct values in this sequence is $ |\NND (O_n) | $.  Set $r=\nnd (I^{\lfloor \frac { n } {2} \rfloor } (p),O_n) $.  Nearest neighbor distances in $O_n $ other than $r $ come from points in $O_n \setminus O_{\lceil \frac { n } {2} \rceil } $ that lie at a distance less than $r $ from $p $.  Thus there must be at least $ |\NND (O_n) - 1 | $ distinct points $I^j(p) $ with $ \lceil \frac { n } {2} \rceil < j \leq n $ and $\dist (p,I^j (p)) <r $.  If $ I^{ j_1} (p) $ and $ I^{j_2} (p) $ are two such points then $$\dist (I^{ j_1} (p), I^{j_2} (p)) = \dist (p,I^{ |j_1 -j_2 | } (p)) \geq r.  $$  Therefore, the number of these points is bounded above by $P (M, r) $, and hence by $P (M) $.
\end{proof}

Observe that only two points can be packed pairwise $r $-apart into an interval of length $2r $, so Lemma \ref{packing} provides a proof of the Geometric Three Gap Theorem.  It is easy to see in general that if $M$ is closed, then $ P (M) $ is finite.  Therefore, any closed Riemannian manifold has an upper bound for the number of nearest neighbor distances occurring in finite orbit segments of isometries.  To prove Theorem \ref{iso1}, however, we must bound $ P (M) $  in terms of geometric data.

\begin{lem}[Bounding the Packing Constant]\label{packingbounds}
Assume that $M^k $ is a Riemannian $k $-manifold with sectional curvatures bounded below by $\kappa $.  Let $M^k_{\kappa } $ be the $k $-dimensional model space of constant curvature $\kappa $.  Then for each $r>0 $, $$P (M,r) \leq P (M^k_{\kappa },r). $$
\end{lem}

\begin{proof}
Consider $n $ points $\{ p_1, \ldots, p_n\} \subset M $ packed pairwise $r $-apart into the open $r $-ball around a point $p \in M $.  We claim that we can also pack $n $ points pairwise $r $-apart into an open $r $-ball in $M^k_{\kappa } $.

Fix a basepoint $\overline{p} \in M_{\kappa}$ and a linear isometry $L:T_{p}M \rightarrow T_{\overline{p}}M_{\kappa}.$  For $i =1,\ldots, n $, let $\gamma_i:[0,1] \rightarrow M$ be a minimizing geodesic in $M$ joining $\gamma_i(0)=p$ to $\gamma_i(1)=p_i$, and $$\overline{\gamma}_i:[0,1] \rightarrow M_{\kappa},\, \, \overline{\gamma}_i(t):=\exp_{\overline{p}} (t \cdot \gamma'_i (0)) $$ the comparison geodesic in $M^k_{\kappa } $. Note that $\dist (\overline{p},\overline{\gamma }_i (1)) = \dist (p,\gamma_i (1)) \leq r $.  Furthermore, by Toponogov's Comparison Theorem (see e.g. \cite[Theorem 2.2]{ChEb}), we have for $i \neq j $ that $$d(\overline{\gamma}_i(1),\overline{\gamma}_j(1))\geq d(\gamma_i(1),\gamma_j(1)) \geq r.  $$  Therefore the points $\{\overline{\gamma }_i (1) \ | \ i =1,\ldots, n\} $ are packed pairwise $r $-apart into the $r$-ball around $\overline{p}\in M^k_{\kappa } $.
\end{proof}

We are now ready to prove the following Theorem, stated previously in the introduction.

\begin{theoremone}
For each $k \in \BN$, $\kappa \in \BR$ and $D>0$, there is a constant $K(k,\kappa,D)\in \BN$ such that for any complete Riemannian $k$-manifold $M^k$ with $\sec \geq \kappa$ and $\diam(M)\leq D,$ and for any $I\in \Isom(M)$, $p \in M$, and $n \in \BN$, $$|\NND(\{I^{i}(p)\,\vert\, i =0,\ldots, n \})| \leq K.$$  
\end{theoremone}
\begin{proof}
It suffices by Lemma \ref{packing} to provide a bound for $P (M) $.  If $\kappa<0$, a trigonometric calculation shows that $P (M^k_{\kappa },r) $ increases monotonically with $r $.  Therefore $P (M) \leq P (M^k_{\kappa },D) $, a constant depending only on $k $, $\kappa $ and $D $.  If $\kappa > 0 $, we have $P(M,r)\leq P(\mathbb{E}^k,r)$, which is independent of $r $ and therefore determined by $k $.
\end{proof}

Bounds on the packing constant of $k $-dimensional Euclidean space can be calculated explicitly.  Recall that a ball of radius $r$ in $\BE^k $ has volume $C(k)r^k$, where $C(k)$ is a dimensional constant.  If $ n $ points are packed pairwise $r$-apart into the $r$ ball around $p \in \BR^k $, then the $r/2$ balls centered at these $ n $ points are disjoint and are all contained in the $3r/2$ ball around $p $.  Hence, $$ n \cdot C(k)\cdot (r/2)^k \leq C(k)\cdot (3r/2)^k,$$ so $P(\BE^k)\leq 3^k $.  This produces the following:
 
\begin{theoremto}
Let $M$ be a complete Riemannian $k$-manifold with non-negative sectional curvatures.  Then for any $I \in \Isom(M)$, $p \in M$, and $n \in \BN$, $$|\NND(\{I^{i}(p)\,\vert\,i =0,\ldots, n\})| \leq 3^k + 1.$$
\end{theoremto}

\subsection{Tori with many distinct nearest neighbor distances}
We now construct orbit segments of translations on flat $k $-dimensional tori with at least $k $ distinct nearest neighbor distances.  The idea is fairly simple.  First, we construct a flat $k $-torus $M$ by gluing together sides of a $k $-dimensional box with side-lengths equal to distinct primes.  This allows us to find a translational isometry of $M$ with an orbit consisting of the entire integer lattice in $M $.  A point $p $ in this orbit is the nearest neighbor to the $2k $ lattice points from which it is distance $1 $.  We can then perturb our isometry slightly so that $k$ of these points now have distinct distances to $p$, while preserving the fact that $p$ is their nearest neighbor.

Fix an orthonormal basis $\{ e_1,\ldots,e_k \} \subset \BR^k $.  Given $v \in \BR^k $, let $ T_v :\BR^k\to \BR^k $ be the translation by $v $.  Let $p_1, \ldots , p_k $ be distinct odd primes, and set $N =\prod_{i = 1}^{k } p_i $.  

Consider the flat torus $$ M =\BR^k / \Gamma, \, \, \Gamma = \left < T_{ p_1 \cdot e_1}, \ldots, T_{ p_k \cdot e_k } \right >.  $$  The translations $T_v $ descend to isometries $t_v : M \to M $, and we denote the projection of a point $p \in \BR^k $ by $\overline{ p } \in M $.

Set $\pi_j \in \{1, \ldots, p_j -1\} $ to be the mod-$p_j $ inverse of $\prod_{i \neq j } p_i $, and let $$a_j = \max (\pi_j,p_j -\pi_j) \cdot \prod_{i\neq j } p_i.  $$  Then $\frac {N - 1 } {2} \leq a_j < N $, $a_j \equiv \pm 1 \, (mod \ p_j) $, and $a_j \equiv 0 \, (mod \ p_i) $ for $i \neq j $. Thus for some $\delta_j \in \{ -1, +1\} $, $\overline {(a_j,\ldots,a_j) } = \overline {\delta_j \cdot e_j } $.

Pick some small $s >0 $, and let $v = (1-s\delta_1,\ldots,1-s\delta_k) $.  Define $ x_i =\overline {i \cdot v } $, and let $X =\{ x_0,\ldots, x_{N - 1 }\}.  $  Note that $X $ consists of the first $N $ elements of the orbit of $\overline {0} $ under the isometry $t_v : M \to M $.  We claim that for $j=1,\ldots,k $, the nearest neighbor in $X $ to $ x_{a_j } $ is $ x_0 $ and the distances $\dist (x_{a_j }, x_0) $ are all distinct.  

First, observe that $ x_{a_j } = \overline {\delta_j\cdot e_j  - a_j s (\delta_1,\ldots,\delta_k) }$, so $$ \dist (x_{a_j }, x_0) = \sqrt{ 1 -2 a_j s + k \left (a_j s \right)^2 }.  $$  Thus if $s $ is small, the distances $\dist (x_{a_j }, x_0) $ are less than $1 $ and are monotonically decreasing with $ a_j $.  In particular, they are distinct, so we need only show that $x_{a_j } $ and $x_0 $ are nearest neighbors in $X $.  

Assume on the contrary that $\dist (x_{a_j }, x_i) < \dist (x_{a_j }, x_0) $ for some $i $.  Then $\dist (x_{ |a_j -i | }, x_0) < \dist (x_{ a_j }, x_0) $ as well.  Since $s $ is small, $x_{ |a_j -i | } $ is very close to the projection of an element of the integer lattice in $\BR^k $.  But $\dist (x_{ |a_j -i | }, x_0) <1 $, so in fact it must be close to $\overline {\pm e_l}$ for some $1\leq l \leq k $.  So either $|a_j -i | = a_l $ or $ |a_j -i | = N - a_l $.  The latter case is impossible since $\dist (x_{N -a_l }, x_0) = \sqrt{ 1 + 2 a_j s + k \left (a_j s \right)^2 }, $ and is therefore greater than $1 $.  If $ |a_j -i | = a_l $, then since $\frac {N -1} {2} \leq a_j < N $ and $1 \leq i < N $ we must have $a_l < a_j $.  This cannot be, because as mentioned earlier $\dist (x_{a_j }, x_0) $ decreases monotonically with $a_j $.

\section{Bounded Geodesic Combinatorics}
\label{geodesicallyflat} 
Recall our definition given in the introduction.

\begin{defi*}[Bounded Geodesic Combinatorics] Let $K\in\BN$.  A complete Riemannian manifold $M$ is defined to have \textit{K-bounded geodesic combinatorics} if for every geodesic $\gamma : \BR \to M $, $T\in\BR$, and $n \in \BN$, $$|\NND(\{\gamma (iT)\,\vert\, i =0,\ldots, n \})| \leq K.$$  
 \end{defi*}

 We begin with some examples.  First, if $M$ is a flat torus then any sequence of equally spaced points along a geodesic in $M$ can be constructed as a segment of an orbit of a translational isometry.  Therefore, Theorem \ref{iso1} shows that flat tori have bounded geodesic combinatorics.  This principle can be applied more generally: 

\begin{prop}
Compact symmetric spaces have bounded geodesic combinatorics.  
\end{prop}

\begin{proof}
Assume that $M$ is a compact Riemannian symmetric space.  Let $\gamma : \BR \to M $ be a unit speed paramaterization of a geodesic, and let $T\in \BR$.  By definintion, the geodesic involution $s_m:M\rightarrow M$ at each point $m \in M$ defined by $\exp_m(w) \mapsto \exp_m(-w)$ for each $w\in TM$ is a well-defined isometry of $M$.  Let $p=\gamma (0) \in M$ and $q=\exp_p(-\frac{T}{2} \dot{\gamma}(0))\in M$.  Define the isometry $I=s_p \circ s_q \in \Isom(M)$.  Then for each natural number $n \in \BN$, $\gamma (nT)=I^n(p)$, and hence Theorem \ref{iso1} implies $M$ has bounded geodesic combinatorics.
\end{proof}

Round spheres are symmetric spaces, and therefore have bounded geodesic combinatorics.  To see this more directly, note that the image of any geodesic in $\BS^k $ is an isometrically embedded circle, so the Three Gap Theorem implies that $\BS^k $ has $3 $-bounded geodesic combinatorics.  This proof motivates considering the following class of spaces.

\begin{defi*}
A closed Riemannian manifold $(M,h)$ is said to be an $SC$-manifold if all its geodesics are simple, periodic and of common least period.  
\end{defi*}  

The Zoll metrics on spheres give explicit non-symmetric examples of $SC$-metrics on spheres of each dimension (see e.g. \cite{Be}).  Note that the definition above does not imply that the images of geodesics are isometrically embedded circles, as is the case for round spheres.  However, the following Lemma allows us to work around this to show that $SC $-manifolds have bounded geodesic combinatorics. Note that the conclusion of the Lemma is equivalent to the existence of a uniform lower bound for the radii of tubular neighborhoods of geodesics in $M$.

\begin{lem}
\label{nicelyembedded}
Let $M$ be a closed Riemannian $SC $-manifold.  Then there exists a constant $r_M>0$ such that for every geodesic $\gamma \subset M $ and points $p,q \in\gamma $ satisfying $\dist_M(p,q)<r_M$, $\dist_M(p,q)=\dist_\gamma (p,q)$. 
\end{lem}

\begin{proof}

Assume this is not the case, and that the common least period of geodesics in $M$ is $L $.  Then we can find a sequence of unit speed geodesics $\gamma_i : \BR \to M $ and times $0 <d_i \leq \frac {L } {2} $ such that $\dist (\gamma_i (0),\gamma_i (d_i)) \to 0 $ but $\gamma_i ([0,d_i]) $ is not minimizing.  Since $d_i $ cannot be less than the injectivity radius of $M$, we can pass to an appropriate subsequence so that $\gamma_n $ converge pointwise to a geodesic $\gamma: \BR \to M $ and $d_i \to d \in [\inj (M) ,\frac {L } {2}] $.  By continuity, $\gamma (0) = \gamma (d) $, contradicting the fact that $\gamma $ is simple with least period $L $.
\end{proof}

\begin{prop}\label{SC}
Suppose that $M$ is a closed Riemannian $SC$-manifold.  Then $M$ has bounded geodesic combinatorics.
\end{prop}
\begin{proof}
Let $\gamma : \BR \to M $, $T>0$, and $n \in \BN$.  As $M$ is compact, there is a number $K_M \in \BN$, so that at most $K_M$ points in $M$ can be pairwise $r_M$ apart, where $r_M$ is as in the previous Lemma.  Therefore, at most $K_M$ of the nearest-neighbor distances $\NND(\{\gamma (iT)\,\vert\, i\in \{0,\ldots,n\}\})$ are not realized as distances measured in $\gamma$.  By the Three Gap Theorem, $$|\NND(\{\gamma (iT)\,\vert\, i\in \{0,\ldots,n\}\})|\leq K_M+3,$$ concluding the proof.
\end{proof}

The argument above can be applied more generally to spaces in which all geodesics are contained in nicely embedded flat tori:

\begin{lem}
\label{geoflat}

Let $M$ be a closed Riemannian manifold.  Assume that there exists a constant $r_M>0$ such that every geodesic $\gamma \subset M$ is contained in an embedded totally geodesic flat torus $F_\gamma \subset M $ with the property that if two points $p,q \in F_\gamma $ satisfy $\dist_M(p,q)<r_M$, then $\dist_M(p,q)=\dist_{ F_\gamma } (p,q)$.  Then $M$ has bounded geodesic combinatorics.
\end{lem}
\begin{proof}
The proof is identical to that of Proposition \ref{SC}, except that it uses the bounded geodesic combinatorics in tori instead of the Three Gap Theorem. 
\end{proof}

A useful quality of Lemma \ref{geoflat} is that if two manifolds satisfy its hypotheses, then their Riemannian product does as well.  It follows from Lemma \ref{nicelyembedded} that $SC $-manifolds satisfy the hypotheses of Lemma \ref{geoflat}.  Any compact symmetric space does as well: its maximal flats are embedded tori and any two of these differ by an isometry of the ambient manifold, \cite[Theorem 6.2]{He}, so any $r_ M $ that works for a single flat works for all flats simultaneously.  This proves the following result.

\begin{kor}
\label{products}
Finite products of $SC $-manifolds and compact symmetric spaces have bounded geodesic combinatorics.
\end{kor}

To appreciate our approach to Corollary \ref{products}, note that it is not obvious, although probably true, that the set of manifolds with bounded geodesic combinatorics is closed under finite products. 

 \section{ Global Consequences of Bounded Geodesic Combinatorics }

This section is devoted to understanding the global behavior of geodesics in manifolds with bounded geodesic combinatorics.

\begin{lem}\label{isolated}
 Let $M$ be a closed manifold with bounded geodesic combinatorics.  Then for every geodesic $\gamma : \BR \to M $, the set of points $\gamma (\BN)$ has finitely many isolated points.
\end{lem}

\begin{proof} 
Let $A \subset \BN $ be the set of indices $i $ for which $\gamma (i ) $ is isolated, and assume that $A $ is infinite.  Set $X = \gamma (\BN) $, and for each $i \in A$, let $L_i = \dist (\gamma (i),X\setminus \gamma (i)) $.  Since $\gamma (i) $ is isolated in $X $, $L_i > 0$.  Note that the $\frac {L_i } {2} $-balls around $\gamma (i) $ are all disjoint, so the compactness of $M $ implies that $L_i \to 0 $.  In particular, if $M$ has $K $-bounded geodesic combinatorics, we can pick $I \subset A$ with $ |I | > K$ such that $L_{i },\ i \in I $ are all distinct.  Set $\epsilon = \min\{ |L_i - L_j | \ | \ i,j \in I, i \neq j\} $.  Then there are indices $n (i), \ i \in I $ with $L_i \leq \dist (\gamma (i),\gamma (n (i))) < L_i + \frac {\epsilon } {2} $.  To finish the proof, set $N = \max \big ( I \cup \{ n (i) \ | \ i \in I\} \big ) $ and $X_N =\gamma (\{0,\ldots,N\}) $.  For each $i \in I $, we have that $L_i \leq \nnd (\gamma (i),X_N) < L_i + \frac {\epsilon } {2} $.  Thus $\nnd (\gamma (i),X_N) $ are all distinct, contradicting the assumption that $M$ has $K $-bounded geodesic combinatorics. 
\end{proof}

\begin{defi*}[Asymptotic ray]
Let $C \subset M$ be a closed geodesic in a complete Riemannian manifold $M$.  A geodesic ray $\gamma:[0,\infty) \rightarrow M$ is defined to be \textit{asymptotic} to $C$ if $$\lim_{t \rightarrow \infty} \dist(\gamma (t), C)=0.$$  We say $\gamma $ is non-trivially asymptotic to $C$ if the image of $\gamma $ is not equal to $C $.
\end{defi*}

\begin{lem}\label{noasymptotic}
Let $M$ be a closed Riemannian manifold with bounded geodesic combinatorics.  Then there are no non-trivial asymptotic rays.
\end{lem}

\begin{proof}
Let $\gamma:[0,\infty) \rightarrow M$ be a geodesic ray non-trivially asymptotic to a closed geodesic $C \subset M$.  Since $\gamma $ and $C $ can intersect only countably many times (see e.g. \cite[Lemma 7.10]{Be}), we can re-parameterize $\gamma $ at a different speed so that $\gamma (i) \notin C,  \ \forall i \in \BN $.  Then since $\gamma $ is asymptotic to $C $, each $\gamma (i) $ is an isolated point of $\gamma (\BN) $.  By Lemma \ref{isolated}, $M$ cannot have bounded geodesic combinatorics.
\end{proof}

\begin{lem}[Geodesic Self Intersections Are Bounded]
\label{intersections}
 Assume that $M$ is a Riemannian manifold that has $K $-bounded geodesic combinatorics.  Let $\gamma : (0, a) \to M $ be a geodesic segment and assume that the image of $\gamma $ has $n $ transverse self intersections of multiplicities $ m_1,\ldots , m_n$.  Then $$\sum_{i=1}^n m_i \leq 2K.  $$  
 \end{lem}

The proof of Lemma \ref{intersections} is somewhat technical, but the idea behind it is geometrically quite clear.  Assume that a geodesic $\gamma \subset M $ intersects itself and points are laid down very finely at regular intervals along $\gamma $.  Then near to each self intersection of $\gamma $ we have a pair of points on different branches of $\gamma $ that are nearest neighbors.  We will show that if our initial set of points was laid down appropriately, then the distances between these pairs of nearest neighbors will all be distinct.

\begin{proof}
Pick some small $\epsilon, \delta >0 $, and let $ X (\epsilon,\delta) = \{ \gamma (i \delta + \epsilon) | i = 0, \ldots, \left\lfloor \frac{ a }{\delta }  - \epsilon \right\rfloor. \} $  Observe that if $\delta $ is small then the distance between a pair of nearest neighbors $\gamma (i \delta + \epsilon),\gamma (j \delta +\epsilon) \in X (\epsilon,\delta) $ is exactly $\delta $ unless $ (i\delta + \epsilon,j\delta +\epsilon) $ lies very close to a pair of times $(s_{i}, t_{j }) $ describing a self intersection of $\gamma $.  Furthermore, near a multiplicity $m$ self intersection point of $\gamma $ there are at least $m$ points in $ X (\epsilon,\delta) $ that lie at a distance less than $\delta $ from some other point of $X (\epsilon,\delta) $: for example, one can take any point in $X (\epsilon,\delta) $ whose distance to the intersection is at most $\frac{\delta}{2} $.  

Let $N \subset \{0,\ldots, \left\lfloor \frac{ a }{\delta }  - \epsilon \right\rfloor\} ^2 $ be a maximal subset such that if $(i,j) \in N $ then
\begin{itemize}
\item $\dist (\gamma (i \delta + \epsilon),\gamma (j \delta + \epsilon)) < \delta $ 
\item $\dist (\gamma (i \delta + \epsilon),\gamma (j \delta + \epsilon))\in \NND(X(\epsilon,\delta))$
\item if $(k,l) \in N $ and $(k,l) \neq (i,j) $, then $\{ s_{i}, t_{ j} \} \neq \{ s_{ k}, t_{ l} \}  $.  \end{itemize}
The last condition means that no two pairs in $N $ lie on the same (unordered) pair of branches of $\gamma $ at the same self intersection.  It is not hard to see that $|N | \geq \frac{1}{2} \sum_{i=1}^n m_i.  $

Without loss of generality, we may assume that $\epsilon$ and $\delta$ have been chosen so that each element of $N$ determines a point of $X (\epsilon,\delta) $ which has a \textit{unique} nearest-neighbor.  Therefore, elements of $N$ still determine pairs of nearest-neighbors in $X (\epsilon+x,\delta(1+y))$ for all choices of $x$ and $y$ sufficiently close to zero.  Hence, the Lemma will be proved if we show that $\epsilon $ and $\delta $ can be perturbed by making appropriate choices of $x$ and $y$ so that the distances $\dist (\gamma (i(\delta+y)+ \epsilon+x),\gamma (j(\delta+y) + \epsilon+x)) , (i,j) \in N $ are all distinct.

Assume that $ s, t \in (0, a) $ and consider for small $c > 0 $ the function $f_{ s, t } : (-c,c) \times (-c,c) \to M $ defined by $$f_{ s, t } (  x  ,  y  ) = \dist (\gamma (s +  x   + s  y  ),\gamma (t +   x   + t  y  )).  $$  Note that the functions $ (f_{s, t })^2 $ vary smoothly with both $s $ and $t $.  If $ s = i \delta +\frac{\epsilon}{1+y} $ and $ t =j\delta +\frac{\epsilon}{1+y} $, then $f_{ s, t } (  x  ,  y  ) $ can be interpreted geometrically as the nearest-neighbor distance realized between the points $\gamma ( i \delta (1+y) + (\epsilon+x) )$ and $\gamma (j\delta (1+y) +(\epsilon+x)))$ from $X (\epsilon +   x  ,\delta ( 1+y))$.

\begin{lem} 
\label{derivatives}
Assume that $\gamma (  s  ) =\gamma (  t  ) = p \in M .  $  Let $\alpha $ be the angle between $\gamma' (  s  ) $ and $\gamma' (  t  ) $ in $ TM_p $.  Then
\begin{eqnarray}
\, \,  \frac{\partial^2}{\partial\xi^2}\big( f_{  s  ,  t  } (\xi, 0) \big)^2 \big|_{\xi = 0} & = & 4-4\cos (\alpha)\nonumber\\
\, \, \frac{\partial^2}{\partial\xi^2} \big( f_{  s  ,  t  }  (0,\xi) \big)^2 \big|_{\xi = 0} & = & 2   s  ^2 + 2  t  ^2 - 4  s    t  \cos (\alpha) \nonumber\\
\ \, \, \frac{\partial^2}{\partial\xi^2} \big( f_{  s  ,  t  }  (\xi,\xi) \big)^2 \big|_{\xi = 0} & = & 2   s  ^2 + 2  t  ^2 - 4  s    t  \cos (\alpha) + \nonumber\\ & & + \left( 4-4\cos (\alpha) \right) \left(1 +   s  +   t  \right). \nonumber \end{eqnarray} 
\end{lem}

\begin{proof}
We will prove the first equality, since the others are proven similarly.  Observe that $f_{  s  ,  t  }  (0, 0) $ and $ \frac{\partial}{\partial\xi} \big( f_{  s  ,  t  } (\xi, 0) \big)^2 |_{\xi= 0} $ both vanish - the latter does because $f_{  s  ,  t  }  (\xi,0) \leq 2\xi $ by the triangle inequality.  It follows that
\begin{eqnarray*}
 \frac{\partial^2}{\partial\xi^2}\big( f_{  s  ,  t  }  (\xi, 0) \big)^2 \big|_{\xi = 0} & = & 2\lim_{\xi \to 0} \frac{(f_{  s  ,  t  }  (\xi))^2}{\xi ^2} \\ &  = & 2\big( \frac{\partial}{\partial\xi} f_{  s  ,  t  }  (\xi, 0) \big |_{\xi = 0^+ } \big)^2.  
\end{eqnarray*}
Now for any $v,w \in  TM_p $ we have $$\dist (\exp_p (v),\exp_p (w)) =  |v -w |  + O ( \max\{ |v |, | w |\}^2 ).  $$  Therefore, $$\frac{\partial}{\partial\xi} f_{  s  ,  t  } (\xi, 0) \big |_{\xi = 0^+ } = \frac{\partial}{\partial\xi}\Big| \exp^{ -1} (\gamma (  s  + \xi)) - \exp^{ -1} (\gamma (  t  +\xi)) \Big |_{\xi = 0^+ }.  $$  Since $\exp^{ -1} (\gamma (  s  + \cdot)) $ and $\exp^{ -1} (\gamma (  t  +\cdot)) $ are simply two lines through the origin in $TM_p $ intersecting with angle $\alpha $, this can be calculated directly using the law of cosines.
\end{proof}

The point of Lemma \ref{derivatives} is that the triples $$D (s, t) = \frac{\partial^2}{\partial\xi^2}\Big (f_{s,  t} (\xi, 0)^2, f_{ s,  t }  (0,\xi) ^2 , f_{  s,  t }  (\xi,\xi) ^ 2\Big ) \Big|_{\xi = 0}  $$ determine the set $\{ s, t\} $ when $\gamma (s) =\gamma (t) $.

We now return to our discussion of $N $.  Recall that by choosing $\delta $ small, we can ensure that if $(i,j) \in N $ then $ (i \delta + \epsilon, j\delta +\epsilon ) $ is arbitrarily close to some pair of times $ (s_i, t_j) $ with $\gamma (s_i) =\gamma (t_j) $.  Also, no two elements of $N $ give the same unordered pair $(s_i, t_j) $.  So, by Lemma \ref{derivatives}, the triples $D (s_i, t_j) $ are distinct for $(i, j) \in N $.  If $\delta $ is chosen small enough, the triples $D (i\delta +\epsilon,j\delta +\epsilon) $ will approximate $D (s_i, t_j) $ very closely, and therefore will be distinct as well.

The terms of $D (i\delta +\epsilon, j\delta +\epsilon) $ are second (directional) derivatives of the function $ (\epsilon,\delta) \mapsto\dist (\gamma (i\delta +\epsilon),\gamma (j\delta +\epsilon)) $.  If two real valued functions differ on some (even higher-order) partial derivative at a point, there is a neighborhood of that point on which the two functions disagree almost everywhere.  Therefore, for almost every small perturbation of $\epsilon $ and $\delta $, the distances $\dist (\gamma (i\delta +\epsilon),\gamma (j\delta +\epsilon)) $ will be distinct for $(i,j) \in N $.  These constitute at least $\frac{1}{2} \sum_{i=1}^n m_i $ nearest neighbor distances in $X (\epsilon,\delta) $, so the Lemma follows.\end{proof}

The final lemma applies only to surfaces, utilizing that geodesics locally separate space:

\begin{lem} 
\label{simpleorclosed}
Let $M$ be a closed surface with bounded geodesic combinatorics.  Then every non-closed geodesic in $M$ is simple and accumulates on itself in $TM $.
\end{lem}
\begin{proof}
Assume that $\gamma \subset M $ is a non-closed geodesic.  As $M$ has bounded geodesic combinatorics, Lemma \ref{isolated} provides a sequence of times $t_n\to \infty $ such that $\gamma (t_n) $ converges to $\gamma (t) $ for some $t \in \BR $.  Moreover, $\gamma $ cannot intersect itself infinitely many times by Lemma \ref{intersections}, so $\gamma '(t_n) $ converges to $\gamma' (t) $ in $TM$. Thus $\gamma $ accumulates on itself in $TM $. Finally, by a continuity argument we see that for arbitrary $s \in \BR $, $\gamma' (t_n + (s - t)) \to \gamma' (s) $ in $TM $.  This implies that $\gamma $ must be simple, for otherwise any self intersection of $\gamma $ will be accompanied by infinitely many other self intersections, which again contradicts Lemma \ref{intersections}.
\end{proof}

\section{Surfaces with Bounded Geodesic Combinatorics}
In this section, we prove Theorem \ref{smooththeorem}. Before giving the proof, we record the following two results about surfaces:  

\begin{sat}[Ga\u{i}dukov, \cite{Ga}]\label{ray}
Let $M^2$ be a closed, oriented surface of positive genus, $p \in M$, and $\Gamma $ a non-trivial free homotopy class of closed cuves in $M$.  Then there is closed geodesic $\gamma \in \Gamma$ and a ray $r:[0,\infty) \rightarrow M$ starting at $p$ that is asymptotic to $\gamma$.  Moreover, any cover of $r$ and any cover of $\gamma$ in the universal cover $\tilde{M}$ is a globally minimizing geodesic.
\end{sat}

\comment {\begin{rmk*} 
When $M^k$ is a closed and non-positively curved Riemannian $k$-manifold, the same conclusion as in Theorem \ref{ray} holds.  This is a standard consequence of the convexity of distance functions in the universal covering.
\end{rmk*} }

\begin{sat}[Innami, \cite{In}]\label{foliation}
Let $M^2$ be a closed, oriented surface of positive genus.  Suppose that for each non-trivial free homotopy class of closed curves $\Gamma$ in $M$, there is a foliation of $M$ by geodesics all belonging to the class $\Gamma$.  Then $M$ is a flat torus.
\end{sat}

\vskip 5pt

We now begin the proof of Theorem \ref{smooththeorem}.

\begin{theoremthree}[Rigidity in Dimension 2]
 Assume that $M$ is a closed Riemannian surface with bounded geodesic combinatorics.  Then $M$ is either
\begin{itemize}
\item $\BT^2 $ with a flat metric,
\item homeomorphic to $\BS^2 $ or $\BR \BP^2 $.
\end{itemize}
In the latter case, if $M$ is real-analytic or positively curved then it is either $\BR\BP^2 $ with a round metric or $\BS^2 $ with an $SC$-metric.
\end{theoremthree}

\begin{proof}
Let $M_O$ denote a connected component of the oriented double cover of $M$.  Consider first the case when $M_O$ has positive genus.  Fix a simple closed curve $\gamma_0 \in M_O $, and let $p \in M_O $.  As in Theorem \ref{ray}, pick a geodesic ray passing through $p$ asymptotic to a closed geodesic $\gamma_p $ in the homotopy class of $\gamma_0 $.  Since $M$ has bounded geodesic combinatorics, Lemma \ref{noasymptotic} implies that the ray's projection in $M$ cannot be nontrivially asymptotic, so its image in $M_O $ is exactly $\gamma_p $.  Therefore, we have a closed geodesic $\gamma_p $ containing $p $ in the homotopy class of $\gamma_0 $.  Recall from Theorem \ref{ray} that $\gamma_p$ lifts to a minimizing geodesic in the universal cover.

We claim that the collection $\{\gamma_p\} $ is a foliation of $M_O $ by circles.  First, each $\gamma_p $ must be simple.  For otherwise, since it can be homotoped to be simple, there must be a pair of arcs on $\gamma_p $ that bound a bigon in $M_O $.  We can then replace one of these arcs by the other and smooth out the corners to create a new curve (based) homotopic to $\gamma_p $ of shorter length, violating the condition that $\gamma_p $ lifts to a distance minimizing geodesic in the universal cover.  Next, assume that $\gamma_p $ and $\gamma_q $ intersect but are not equal.  Since they can be homotoped to be disjoint, there is an arc on each which when combined bound a bigon in $ M_O $.  This is a contradiction as before.  

Since $M_O $ is orientable and can be foliated by circles, it is topologically a torus.  Each free homotopy class of closed curves on the torus admits a simple representative.  Therefore, repeating the argument above proves that for each homotopy class there is a foliation of $M_O$ by closed geodesics in that class.  Theorem \ref{foliation} now implies that $M_O $ is flat.  Therefore, $M$ is either a flat torus or a flat Klein bottle.  It is easy, however, to find geodesics on a flat Klein bottle that intersect themselves arbitrarily many times, so since $M$ has bounded geodesic combinatorics, Lemma \ref{intersections} implies that $M$ is a flat torus.

Assume now that $M$ is homeomorphic to either $\BS^2 $ or $\BR \BP^2 $, and that $M$ is either real-analytic or positively curved.  It suffices to show that all geodesics in $M$ are closed.  Indeed, Gromoll and Grove \cite{GrGr} have proven that any Riemannian metric on $\BS^2 $ with all geodesics closed is in fact an $SC$-metric, and Pries \cite{Pr} showed that a Riemannian metric on $\BR \BP^2 $ has all its geodesics closed if and only if it has constant curvature.  

Fix a point $p \in M $.  For $v \in SM_p$, let $\gamma_v:\BR \rightarrow M$ denote the unit speed paramaterization of the geodesic with $\dot{\gamma_v}(0)=v$. 
  Define $$U =\{ v \in SM_p \ | \ \exists \ T>0 \ \text{such that} \ \gamma_v(T) \ \text {is conjugate to } p \}.  $$ Since $M$ has bounded geodesic combinatorics, Lemma \ref{simpleorclosed} states that every non-closed geodesic in $M$ is simple and accumulates on itself in $TM$.  This fact combined with the following Lemma implies that for $v \in U$, $\gamma_v $ is closed.

\begin{lem}
\label{conjugateimpliesintersection}
Assume that $M$ is a Riemannian surface, $\gamma :\BR\to M $ is a geodesic, and $\gamma (a) $ and $\gamma (b) $ are conjugate along $\gamma $.  Then given $\epsilon >0 $, there exists $\delta > 0 $ so that if $\alpha :\BR \to M $ is a geodesic with $$|\alpha (t) -\gamma (t) | <\delta, \ \forall  t \in [a -\epsilon,b +\epsilon] $$ then $\alpha ([a -\epsilon,b +\epsilon]) $ must intersect $\gamma ([a -\epsilon,b +\epsilon]) $.
\end{lem}

\begin{proof}
Assume without loss of generality that there is no point on $\gamma $ before $\gamma (b) $ that is conjugate to $\gamma (a) $.  Extend $\gamma $ to a geodesic variation $\gamma_t : [a -\epsilon,b +\epsilon] \to M $ with $\gamma_0 = \gamma $ such that the Jacobi field $\vec {v } (s) = \frac{\partial}{\partial t}\gamma_t(s) $ is nontrivial, orthogonal to $\gamma' (s) $ and vanishes at $s= a $ and $s=b$, but not between.  Because $\vec {v } $ is not identically zero and determined by the covariant derivative $\frac {\DD} {ds }\vec {v } (s ) |_{ s = a } $, this derivative must be nonzero.  Thus $\vec {v } (s) $ must lie for $s >a $ on the side of $\gamma' (s) $ opposite to that on which it lies for $s <a $.  A similar statement holds for $s $ near $b $.  Therefore, by considering only $ t $ sufficiently close to $0 $, we can ensure that each $\gamma_t $ intersects $\gamma$ very close to $\gamma (a) $ and $\gamma (b) $.  Moreover, $\gamma_t $ cannot intersect $\gamma $ far away from $\gamma (a) $ and $\gamma (b) $ because $\vec {v } (s) $ vanishes only when $s = a $ or $s =b $.  Finally, since the intersections between two geodesics cannot be too close together we see that $\gamma $ and $\gamma_t $ intersect exactly twice.  Thus the middle portion of each $\gamma_t $ lies on one side of $\gamma $ and the ends lie on the other side.  

Choose $\delta $ small enough so that if $\alpha $ is as in the statement of the Lemma then it intersects the middle portion of some geodesic in the variation $\gamma_t $.  Assume that $\alpha ([a -\epsilon,b +\epsilon]) $ does not intersect $\gamma ([a -\epsilon,b +\epsilon]) $ and let $\gamma_{t_0 } $ be the first geodesic in the variation that $\alpha([a -\epsilon,b +\epsilon])  $ does intersects.  The intersection of $\alpha $ and $\gamma_{ t_0} $ must lie in the interior of both segments, since close to their endpoints the segments lie on opposite sides of $\gamma $.  Because $\alpha $ does not intersect $\gamma_t $ for $t <t_0 $, its intersection with $\gamma_{ t_0} $ cannot be transverse.  This is a contradiction, because geodesics always intersect transversely.
\end{proof} 

Our goal, then, is to show that $U = SM_p$.  When $M$ has positive curvature, this is a simple consequence of the Rauch comparison theorem (see e.g. \cite[Theorem 2.3, Chapter 10]{do Ca}).  So, from now on we assume that $M$ is real-analytic.  Since $\tilde{M}$ is compact, $U $ must be nonempty.  Moreover, a result of Warner \cite[Theorem 3.1]{Wa} implies that each component of the tangential conjugate locus of $M$ at $p $ is a properly embedded $1$-submanifold of $TM_p $ transverse to the radial direction.  Since $U $ is the image of the tangential conjugate locus under the radial projection $ TM_p \setminus \{0\} \to SM_p $, $U $ is an open non-empty subset of $SM_p $.  Seeking a contradiction, assume that $I\neq SM_p$ and let $I \subset U $ be a maximal open interval.

\begin{lem}
\label{commonperiod}
If $C \subset I$ is a compact subinterval, then the geodesics $\{ \gamma_v \ | \ v \in C \} $ have a common period.
\end{lem}

The proof of Lemma \ref{commonperiod} proceeds in several steps.  We will first show that for $v\in C $ there is an upper bound for the least period of $\gamma_v $.  Next, we prove that there is a continuous function $ l: C \to \BR $ such that $ l (v) $ is a period for $\gamma_v $, and then prove that the function is smooth.  In the final step, we use $l $ to parameterize the family $\gamma_v $ as a proper geodesic variation and deduce from the first variational formula that in fact $ l $ is constant.

\begin{proof}[Proof of Lemma \ref{commonperiod}]

For each $v \in C $, let $lp (v) $ be the least period of $\gamma_v $.  We first show that $lp (v) $ is bounded on $C $.  Assume on the contrary that there is a sequence $v_n\in C $ with the property that $lp (v) \to \infty $.  Without loss of generality, we may assume that $v_n \to v \in C $.  Define $\hat{\gamma}_v,\hat{\gamma}_{v_n} : \BR \to SM $ to be lifts of $\gamma_v $ and $\gamma_{ v_n } $ to the unit tangent bundle of $M$, and let $N $ be a small regular neighborhood of $\hat{\gamma}_v (\BR) \subset SM $.  

Since $v \in C $, there is some $T >0 $ so that $\gamma_v (0) $ and $\gamma_v (T) $ are conjugate along $\gamma_v $.  Assume that $\beta \subset M$ is a geodesic segment that lifts into $N $ and has length at least $K =T +lp (v) $.  If $N $ is small, $\beta $ has a subsegment that closely tracks $\gamma_v ([0,T]) $.  Lemma \ref{conjugateimpliesintersection} implies that if this subsegment stays close enough to $\gamma_v ([0,T]) $, then $\beta $ must in fact intersect $\gamma_v $.   Therefore we can choose $N $ small enough so that any such geodesic segment $\beta $ must intersect $\gamma_v $.

Assume that $n $ is very large.  Then $\hat{\gamma}_{ v_n } $ spends at least a duration of $K $ inside of $N $.  Moreover, $\hat{\gamma }_{ v_n } $ must eventually exit $N $.  For otherwise, $\gamma_{ v_n } $ will be a closed curve homotopic within an annular neighborhood of $\gamma_v $ to a large power of $\gamma_v $, and therefore will intersect itself more often than is allowed by Lemma \ref{intersections}.  Set $t_n $ to be the first time at which $\hat{\gamma }_{ v_n } $ exits $N $, and let $g_n =\gamma_{ v_n } ([t_n -K, t_n]) $.  Since $g_n $ lifts into $N $ and has length $K $, it must intersect $\gamma_v $.  If the angle of intersection is very small, then $g_n $ will be forced to track $\gamma_v $ for a long time, and therefore cannot exit $N $ before a duration of $K $ had elapsed.  Thus there is some $\alpha >0 $ such that each $g_n $ intersects $\gamma_v $ with angle at least $\alpha $.  But if $n $ is very large, $\gamma_{ v_n } ([0, t_n -K]) $ winds many times around $\gamma_v $ while staying close enough to pass through any geodesic segment intersecting $\gamma_v $ with angle at least $\alpha $.  Therefore, $\gamma_{ v_n } ([0, t_n -K]) $ must intersect $ g_n =\gamma_{ v_n } ([t_n -K, t_n]) $ many times as well.  This is impossible, since the number of self intersections of $\gamma_{ v_n } $ is limited by Lemma \ref{intersections}.  

We now know that $lp (v) $ has an upper bound on $C $, which we set to be $L $.  One geometric consequence of this is that a convergent sequence $v_n \to v\in C $ gives a sequence of geodesics $\gamma_{ v_n } $ whose images converge to the image of $\gamma_v $ in the Hausdorff topology on closed subsets of $M $.  Note that for large $n $, this implies that $lp (v_n) $ is very close to a period of $\gamma_v $.

Our aim is now to find a continuous function $ l: C \to \BR $ with $ l (v ) $ a period for $\gamma_v $.  We will start by defining it locally.  So, fix a vector $ w\in C $; our goal is to produce a neighborhood $N_w \subset C$ of $w $ and a continuous function $l_w : N_w \to \BR $ that gives periods for $\gamma_v, \ v \in N_w $.  To begin with, set $$l_w (w) = \left\lfloor \frac {L } {lp (w) } \right\rfloor  \ lp(w) .  $$  Define $N_w $ to be a neighborhood of $w $ small enough so that if $v \in N_w $, then $lp (v) $ is within $\frac {lp (w)} {8 \left\lfloor \frac {L }  {lp (w) }\right\rfloor  } $ of a period of $\gamma_w $.   Then for $v \in N_w $, set
\begin{center}
$ l_w (v) = $ the period of $\gamma_v  $ that is within $\frac {lp (w) } {8}  $ of $l_w (w) .  $  
\end{center}

We claim that $l_w : N_w \to \BR $ is continuous.  Assume that $v_n \to v \in N_w $.  By definition, $l_w (v) $ and each $l_w (v_n) $ are within $\frac {lp (w) } {8} $ of $ l_w (w) $.  Therefore $| l_w (v) - l_w (v_n) | < \frac {lp (w) } {4} $.  Since $$lp (v) > lp(w) - \frac {lp (w)} {8 \left\lfloor \frac {L }  {lp (w) }\right\rfloor   } > \frac { lp(w) } {2}, $$ this implies $$| l_w (v) - l_w (v_n) | < \frac {lp (v) } {2}.  $$  So, $l_w (v_n) $ is closer to $l_w (v) $ than to any other period of $\gamma_v $.  As in the previous paragraph, the fact that least periods are bounded in $C $ implies that when $n $ is large, $l_w (v_n) $ is very close to a period of $\gamma_v $.  This period must then be $l_w (v) $.  Therefore, $l_w (v_n) \to l_w (v) $.  

We now have for each $w \in C$, a continuous function $l_w: N_w \to \BR $ such that $l_w (v) $ is a period for $\gamma_v $.  Pick a finite set $\{ w_1,\ldots, w_n \} \subset C $ so that $N_{ w_i } \cap N_{ w_{i +1} } \neq \emptyset $ and $\cup_{i = 1}^{ n } N_{ w_i } $ covers $C $.  Since $ l_{ w_1 } $ and $l_{ w_2} $ are continuous and their quotient is rational valued, they can be multiplied by appropriate positive integers so that they coincide on the intersection of their domains.  Note that this does not change the property that they pick out periods for $\gamma_v , \ v \in C $.  A similar trick applies to make $l_{ w_3} $ agree with the previous two.  Continuing inductively, we can piece our locally defined functions together to create a continuous function $l : C \to \BR $.

To show that $l $ is smooth, pick a vector $w\in C $ and choose a coordinate chart $\phi : O \to \BR^2 $ with $\phi (p) = (0, 0) $ and $\derivative\phi_p (\gamma_w' (l (w)w)) = (1, 0) $.  Let $\pi : \BR^2 \to \BR $ be the projection onto the first coordinate.  If $V \subset TM_p $ is a small neighborhood of $l (w) w $, then $\pi \circ \phi \circ \exp_p: V \to \BR $ is defined and a submersion.  By the Implicit Function Theorem, $(\pi \circ \phi \circ \exp_p )^{ - 1} ( 0) $ is a smooth $1 $-submanifold of $V $.  Since $\{l (v) v\ | \ v \in C \} \cap V $ is a continuous $1 $-manifold contained in that preimage, the connected components of both that contain $l (w) w $ must coincide.  This shows that $l $ is smooth in a neighborhood of $w \in C $.

Finally, define a smooth geodesic variation $G : C \times [0, 1] \to M $ by $G (v, t) = \gamma_v(l (v) t) $. Then $G (v,0) = G (v,1) = p $ for all $v \in C $, and the length of each segment $G (v, [0, 1]) $ is $l (v) $.  By the First Variational Formula, $\frac{\partial}{\partial v } l (v) = 0 $.  So $l $ is constant.
\end{proof}

To finish the proof, fix some compact subinterval $C \subset I $ and assume that $L $ is a common period for $\{ \gamma_v \ | \ v \in C \} $.  Since $M$ is real-analytic, the square of the Sasaki metric $\dist_{ TM }^2 : TM \times TM \to \BR $ is as well.  Define a function $f : I \to \BR $ by $$f (v) = \dist_{ TM }^2 \Big([\gamma_v (L),\gamma_v' (L)], [p,\gamma_v' (0)]\Big).  $$  Then $f $ is real-analytic and vanishes on $C $, so it must be identically zero.  This shows that $L $ is a common period for $\{\gamma_v \ | \ v \in I \} $.  But then if $v $ is an endpoint of $I $, $L $ must also be a period for $\gamma_v $.  Furthermore, $\gamma_v $ is part of a (one-sided) geodesic variation all of whose geodesics start at $p $ and close up at time $L $.  Therefore $\gamma_v $ contains a point conjugate to $p $. This is a contradiction, by definition of $I $.
\end{proof}

\subsection{Remarks} 
The definition of bounded geodesic combinatorics is never used explicitly in Theorem \ref{smooththeorem}.  In fact, what we have done is characterize those (real-analytic) Riemannian surfaces on which
\begin{enumerate}
\item every non-closed geodesic accumulates on itself in $TM $ 
\item there is an upper bound for the number of times a geodesic segment can intersect itself.
\end{enumerate}
It would be interesting to see if a similar characterization holds in the presence of only one of these conditions; assuming the latter condition seems particularly likely to result in the same conclusion.

In a similar vein, parts of Theorem \ref{smooththeorem} can be proven differently by making more judicious use of the geodesic self-intersection bounds given by Lemma \ref{intersections}.  For instance, higher genus surfaces cannot support a metric with bounded geodesic combinatorics because the topology forces geodesics in certain homotopy classes to self intersect many times.  The case in which $M$ is positively curved can also be handled as follows.  Since $M$ is compact, it has a lower curvature bound $\kappa > 0 $.  By \cite{VaMa}, no geodesic arc in $M$ of length greater than or equal to $\frac{3\pi}{\sqrt{\kappa} }$ can be simple.  Every geodesic in $M$ must then be closed, for otherwise it would have infinitely many self intersections, which is prohibited by Lemma \ref{intersections}.  As mentioned in the proof above, this implies that $M $ is either a round $\BR\BP^2 $ or $\BS^2 $ with an $ SC $-metric.

It is likely that the conclusion of Theorem \ref{smooththeorem} holds without assuming that $M$ is positively curved or real-analytic.  Our argument applies in the smooth case through Lemma \ref{commonperiod}, but we are not sure how to show without using real-analyticity that the least period of $\gamma_v $ does not approach infinity as $v $ tends towards the endpoints of $I$.  We are confident that if this happens the number of self intersections of $\gamma_v $ should also approach infinity, but writing down an argument has proven difficult.

{\tiny \sc Ian Biringer, University of Chicago, biringer@math.uchicago.edu}

{\tiny \sc Benjamin Schmidt, University of Chicago, schmidt@math.uchicago.edu}

\end{document}